\documentclass[12pt, reqno]{amsart}
\usepackage{amsmath}
\usepackage{amssymb}
\usepackage{epsfig}
\usepackage{graphicx}
\usepackage[top=30truemm,bottom=30truemm,left=30truemm,right=30truemm]{geometry}
\usepackage{color}
\usepackage{comment}
\definecolor{shadecolor}{gray}{0.875}
\usepackage{amscd}
\usepackage{stmaryrd}
\usepackage{blindtext}
\usepackage{hyperref}

\numberwithin{equation}{section}

\input xy
\xyoption{all}

\calclayout
\allowdisplaybreaks[3]

\theoremstyle{plain}
\newtheorem{prop}{Proposition}[section]

\newtheorem{theo}[prop]{Theorem}
\newtheorem{coro}[prop]{Corollary}

\newtheorem{lemm}[prop]{Lemma}

\theoremstyle{definition}

\newtheorem{conj}[prop]{Conjecture}

\newtheorem{rema}[prop]{Remark}

\def\rH{{\mathrm H}}

\def\Nef{\mathrm{Nef}}

\def\Spec{\mathrm{Spec}}

\makeatother
\makeatletter

\author{Sho Tanimoto}
\address{Graduate School of Mathematics, Nagoya University, Furocho Chikusa-ku, Nagoya, 464-8602, Japan}
\email{sho.tanimoto@math.nagoya-u.ac.jp}

\title[Homological sieve and Manin's conjecture]{Homological sieve and Manin's conjecture}

\begin{document}
\date{\today}

\begin{abstract}
This is a report of the author's talk at RIMS workshop Algebraic Number Theory and Related Topics 2025 which was held at RIMS Kyoto University during December 15th--19th 2025.
In this survey paper, we explain the homological sieve method, which is proposed by Das, Lehmann, Tosteson, and the author in \cite{DLTT25}, and its applications to Manin's conjecture.
\end{abstract}

\maketitle

\section{Introduction}

One of the central topics in diophantine geometry is the distribution of rational points on algebraic varieties defined over global fields. Here a main question is the asymptotic formula for the counting function of rational points of bounded height, and one major conjecture is Manin's conjecture which predicts the asymptotic formula for the counting function of rational points of bounded height on a smooth Fano variety defined over number fields. This conjecture was originally initiated by Yuri Manin and subsequently developed in \cite{FMT89}, \cite{BM}, \cite{Peyre}, \cite{BT}, \cite{Peyre03}, \cite{Peyre16}, \cite{LST18}, and \cite{LS24}.

One can also consider Manin's conjecture over global function fields. To this end, Batyrev developed a heuristic argument in \cite{Bat88}, and Ellenberg and Venkatesh suggested that one should use homological stability of the moduli spaces of curves to realize Batyrev's heuristics (see, e.g., \cite{EVW16}). Since then people have been seeking a topological proof of Manin's conjecture over global function fields, however such a proof was only available for easy examples such as projective spaces.

In 2024, Das and Tosteson uploaded \cite{DT24} on arXiv, and in this paper, the authors established a version of Cohen--Jones--Segal conjecture for quintic del Pezzo surfaces which might be considered as a topological version of Manin's conjecture. After this preprint, Das, Lehmann, Tosteson, and the author started to work together, and we established a version of Manin's conjecture for rational curves on quartic del Pezzo surfaces in \cite{DLTT25}. In these papers, the authors developed a version of homological sieve method which encodes the inclusion-exclusion principle in its core. To our knowledge, this was the first example of a topological proof of Manin's conjecture over global function fields. This homological sieve method has been further explored in \cite{Tan25} and \cite{TT25}.
In this survey paper, we will introduce this homological sieve method and its applications to Manin's conjecture and Cohen--Jones--Segal conjecture.

\

\noindent
{\bf Acknowledgements:}
The author would like to thank Ronno Das, Brian Lehmann, and Phil Tosteson for collaborations helping to shape his perspective on homological methods in Manin's conjecture.
The author would like to thank the organizers of RIMS workshop Algebraic Number Theory and Related Topics for the opportunity to give a talk there. This work was supported by the Research Institute for Mathematical Sciences, an International Joint Usage/Research Center located in Kyoto University.
The author thanks Haruki Ito and Natsume Kitagawa for comments on an early draft of this paper.

Sho Tanimoto was partially supported by JST FOREST program Grant number JPMJFR212Z and by JSPS KAKENHI Grant-in-Aid (B) 23H01067.

\section{Manin's conjecture}

Let $k$ be a number field and $X$ be a smooth Fano variety defined over $k$, i.e., $X$ is projective and the anticanonical divisor $-K_X$ is ample. 
We fix an adelic metrization for the anticanonical bundle $\omega_X^{-1} = \mathcal O(-K_X)$. This induces the height function, which is a real-valued function on the set of rational points, i.e., 
\[
\mathsf H_{-\mathcal K_X} : X(k) \to \mathbb R_{\geq 0}.
\]
This measures the geometric and arithmetic complexity of rational points, and it is the fundamental object in arithmetic geometry.
One remarkable feature of this function is that it satisfies the {\it Northcott property}, i.e., the set of rational points of bounded height is a finite set. Thus we define, for any subset $Q \subset X(k)$, its counting function:
\[
\mathsf N(Q, -\mathcal K_X, T) = \# \{ x \in Q \, |\, \mathsf H_{-\mathcal K_X}(x) \leq T\}.
\]
Manin's conjecture, which predicts the asymptotic formula for this counting function, was originally formulated mainly by Batyrev--Manin--Peyre--Tschinkel, and subsequently refined in a series of work \cite{FMT89, BM, Peyre, BT, Peyre03, Peyre16, LST18, LS24}:

\begin{conj}[Manin's conjecture]
\label{conj:Manin}
Suppose that $X(k)$ is not thin.
Then there exists a thin set $\mathsf Z \subset X(k)$ such that
\[
\mathsf N(X(k) \setminus \mathsf Z, -\mathcal K_X, T) \sim \alpha(\mathrm{Nef}_1(X)) \beta(X) \tau_{-\mathcal K_X}(X) T(\log T)^{\rho(X)-1},
\]
as $T \to \infty$ where $\rho(X)$ is the Picard rank of $X$, $\alpha(\mathrm{Nef}_1(X))$ is the alpha constant of the nef cone $\mathrm{Nef}_1(X)$ of curves, $\beta(X)$ is the size of $\mathrm{Br}(X)/\mathrm{Br}(k)$, and $\tau_{-\mathcal K_X}(X)$ is the Tamagawa number, introduced by Peyre for the anticanonical divisor (\cite{Peyre}) and Batyrev--Tschinkel for arbitrary big divisors (\cite{BT}).
\end{conj}

\begin{rema}
A thin set is any subset of a finite union $\cup_i f_i(Y_i(k))$ where $f_i : Y_i \to X$ is a generically finite morphism to the image from a variety $Y_i$, but not birational to $X$. The thin set $\mathsf Z$ is called the {\it exceptional set}, and it is important to remove the contribution from this set because it is possible that rational points are accumulating along subvarieties. Originally it was expected that $\mathsf Z$ should be non-Zariski dense, however, the first counterexample to this expectation was found in \cite{BT-cubic}. Peyre was the first to suggest that the exceptional set should be a thin set (\cite{Peyre03}). A conjectural characterization of $\mathsf Z$ was obtained in \cite{LST18} by Lehmann, Sengupta, and the author, and we even proved that the proposed set is indeed a thin set using higher dimensional algebraic geometry such as the minimal model program.
For more details of Conjecture~\ref{conj:Manin}, see, e.g., \cite{Tani26} and references therein.
\end{rema}

Conjecture~\ref{conj:Manin} is largely open even in dimension $2$, i.e., when $X$ is a {\it del Pezzo surface}. Let us explain the status of this conjecture for surfaces:
\begin{itemize}
\item del Pezzo surfaces of degree $\geq 6$ are toric surfaces. Conjecture~\ref{conj:Manin} for projective toric varieties was proved by Batyrev--Tschinkel in \cite{BT98, CLT01};
\item split quintic del Pezzo surfaces over $\mathbb Q$ in \cite{delaBre02} and over arbitrary number fields in \cite{BD25};
\item some non-split quintic del Pezzo surfaces over $\mathbb Q$ in \cite{HBL25}; and
\item An example of a quartic del Pezzo surface over $\mathbb Q$ in \cite{dBB11}.
\end{itemize}

In particular, there is no single example of a smooth cubic surface for which we know a proof of Conjecture~\ref{conj:Manin}. In this survey paper, we would like to explore Conjecture~\ref{conj:Manin} for del Pezzo surfaces over global function fields, e.g., $\mathbb F_q(t)$.

\section{The space of rational curves on del Pezzo surfaces}
Let $k$ be a field and $K = k(t)$ be the function field of $\mathbb P^1$.
Let $X$ be a projective variety defined over $k$.
By valuative criterion, a $K$-rational point $\Spec \, K \to X$ corresponds to a section $\sigma : \mathbb P^1 \to X \times \mathbb P^1$ of the trivial family, which in turn corresponds to a rational curve $s : \mathbb P^1 \to X$.
This observation naturally leads us to consider the space of rational curves: let $X$ be a projective variety defined over $k$, and $\alpha$ be a numerical class of $1$-cycles on $X$. Then the space of rational curves is the morphism scheme:
\[
\mathrm{Mor}(\mathbb P^1, X, \alpha),
\]
which parametrizes rational curves $s : \mathbb P^1 \to X$ such that $s_*[\mathbb P^1] = \alpha$.
This is a quasi-projective scheme over $k$, and it is a fine moduli scheme of such rational curves. 

Assume that $k = \mathbb F_q$ is a finite field and $S$ is a smooth del Pezzo surface. The counting function of rational curves on $S$ is defined as follows:
\[
\mathsf N(\mathbb P^1, S, -K_S, d) = \sum_{\alpha \in \mathrm{Nef}_1(S)_{\mathbb Z}, -K_S.\alpha \leq d} \#\mathrm{Mor}(\mathbb P^1, S, \alpha)(k),
\]
where $\mathrm{Nef}_1(S)$ is the nef cone of curves on $S$ and $\mathrm{Nef}_1(S)_{\mathbb Z}$ is the set of lattice points in $\mathrm{Nef}_1(S)$. This is the number of $k$-nef rational curves of anticanonical degree $\leq d$.

Manin's conjecture predicts the asymptotic formula:
\[
\mathsf N(\mathbb P^1, S, -K_S, d) \sim (1-q^{-1})\alpha(\mathrm{Nef}_1(S))\beta(S)\tau_{-K_S}(S) q^dd^{\rho(S)-1},
\]
as $d \to \infty$. How one can prove such a statement? One natural approach to this, proposed by Batyrev and Ellenberg--Venkatesh, is to use the Grothendieck--Lefschetz trace formula for $\#\mathrm{Mor}(\mathbb P^1, S, \alpha)(k)$, i.e., 
\[
\# \mathrm{Mor}(\mathbb P^1, S, \alpha)(\mathbb F_q) = \sum_i (-1)^i \mathrm{Tr}(\mathrm{Frob} : \mathrm{H}^i_{\text{\'et}, c} (\mathrm{Mor}(\mathbb P^1, S, \alpha)_{\overline{k}}, \mathbb Q_\ell)).
\]
To this end, we need to establish certain stability of \'etale cohomology
\[
\mathrm{H}^i_{\text{\'et}, c} (\mathrm{Mor}(\mathbb P^1, S, \alpha)_{\overline{k}}, \mathbb Q_\ell).
\]
This motivates topological counting arguments using the homological stability of the moduli space, pioneered in \cite{EVW16}.

\subsection*{The main theorem of \cite{DLTT25}}
The first (?) kind of topological proofs for Manin's conjecture over global function fields was developed by Das, Lehmann, Tosteson, and the author in \cite{DLTT25}, and in this paper, we proposed a homological sieve method, which is a combination of birational geometry (the birational geometry of the space of rational curves), arithmetic geometry (simplicial schemes and the Grothendieck--Lefschetz trace formula),  algebraic topology (the inclusion-exclusion principle and Vassiliev type bar complex), and elementary analytic number theory.
Let us explain the main theorem of \cite{DLTT25}.

Let $k = \mathbb F_q$ as before and $S$ be a split smooth quartic del Pezzo surface defined over $k$. Here split means that we have $\rho(S) = \rho(S_{\overline{k}})$.
Let $\ell$ be a non-negative, rational, homogeneous, continuous, and piecewise linear functional on $\mathrm{Nef}_1(S)$.
Let $\epsilon > 0$ be a small positive rational number. We consider the shrunken cone:
\[
\mathrm{Nef}_1(S)_\epsilon = \{ \alpha \in \mathrm{Nef}_1(S) \, |\, \ell(\alpha) \geq -\epsilon K_S.\alpha \}.
\]
With this cone, we consider the following modified counting function:
\[
\mathsf N(\mathbb P^1, S, -K_S, \epsilon, d) = \sum_{\alpha \in \mathrm{Nef}_1(S)_{\epsilon, \mathbb Z}, -K_S.\alpha \leq d} \#\mathrm{Mor}(\mathbb P^1, S, \alpha)(k).
\]
Here is the main theorem of \cite{DLTT25}:

\begin{theo}[{\cite[Theorem 1.1]{DLTT25}}]
\label{theo:main}
Let $C = 2^{32}$ and assume that $q^\epsilon > C$.
There exists a non-negative, rational, homogeneous, continuous, and piecewise linear functional $\ell$ on $\mathrm{Nef}_1(S)$ such that
$\ell$ is positive on a dense open cone $U \subset \mathrm{Nef}_1(S)$ and we have
\[
\mathsf N(\mathbb P^1, S, -K_S, \epsilon, d) \sim (1-q^{-1})\alpha(\mathrm{Nef}_1(S)_\epsilon)\tau_{-K_S}(S) q^dd^{\rho(S)-1},
\]
as $d \to \infty$. Here the Tamagawa number is given by the Euler product:
\[
\tau_{-K_S}(S) = q^2(1-q^{-1})^{-6} \prod_{c \in |\mathbb P^1|} (1-q^{-|c|})^{6} \frac{\#S(\mathbb F_{q^{|c|}})}{q^{2|c|}},
\]
where $|\mathbb P^1|$ is the set of closed points on $\mathbb P^1$ and $|c|$ is $[k(c): k]$.
\end{theo}
In particular, this gives a lower bound of the correct magnitude for $\mathsf N(\mathbb P^1, S, -K_S, d).$
In this survey paper, we explain the proof of the above theorem using homological sieve method.

\section{Birational geometry of the moduli space of rational curves}
\label{sec:birationalmoduli}

Let $k$ be a field and $S$ be a split smooth quartic del Pezzo surface defined over $k$.
Let $B = \mathbb P^1$.
Let $\alpha \in \mathrm{Nef}_1(S)_{\mathbb Z}$ be a nef class on $S$. We would like to understand the structure of the morphism scheme
\[
\mathrm{Mor}(B, S, \alpha).
\]
First of all, we know that this is geometrically irreducible and has expected dimension $-K_S.\alpha + 2$ by \cite{Testa09} and \cite{BLRT21}.
Then a key observation to the geometry of this scheme is the following lemma:
\begin{lemm}[{\cite[Proposition 3.3 and Notation 3.5]{DLTT25}}]
\label{lemm:birationalmor}
For each nef class $\alpha \in \mathrm{Nef}_1(S)_{\mathbb Z}$, there exists a birational morphism
\[
\rho_\alpha : S \to \mathbb P^1 \times \mathbb P^1,
\]
with the following properties: let $E_1, \cdots, E_4$ be the exceptional divisors contracted by $\rho_\alpha$ and $F, F'$ are general fibers of two conic fibrations coming from $\mathbb P^1 \times \mathbb P^1$. Then we have
\[
2F.\alpha - \sum_{i = 1}^4 E_i.\alpha \geq 0, \quad 2F'.\alpha - \sum_{i = 1}^4 E_i.\alpha \geq 0.
\]
\end{lemm}

From now on, we fix such a birational morphism $\rho_\alpha$ for each $\alpha$, and we set
\[
a(\alpha) = F.\alpha,\, a'(\alpha) = F'.\alpha,\,  k_i(\alpha) = E_i.\alpha, \, \text{ and } h(\alpha) = -K_S.\alpha.
\]
Note that we have
\[
h(\alpha) = 2a(\alpha) + 2a'(\alpha) - \sum_{i = 1}^4k_i(\alpha).
\]

Next we consider the following Zariski open subset of the product of Hilbert schemes:
\[
U_{\bold k(\alpha)} := \{ (T_i) \in \prod_{i = 1}^4 \mathrm{Hilb}^{[k_i(\alpha)]}(B) \, | \, \text{the supports of the $T_i$'s are mutually disjoint} \},
\]
where $\bold k(\alpha) = (k_1(\alpha), k_2(\alpha), k_3(\alpha), k_4(\alpha))$.
Then we define the following structure morphism which is a key geometric structure for our counting arguments:
\[
\Phi_{\alpha} : \mathrm{Mor}(B, S, \alpha) \to U_{\bold k(\alpha)}, [s: B \to S] \mapsto (s^*E_i)_{i = 1}^4.
\]
Regarding this morphism, we have
\begin{prop}[{\cite[Lemma 4.4 and Lemma 4.9]{DLTT25}}]
\label{prop:structuretheorem}
The morphism $\Phi_{\alpha}$ is dominant and the image $\Phi_{\alpha}(\mathrm{Mor}(B, S, \alpha))$ is a Zariski open subset in $U_{\bold k(\alpha)}$. Moreover,
\[
\Phi_{\alpha} : \mathrm{Mor}(B, S, \alpha) \to \Phi_{\alpha}(\mathrm{Mor}(B, S, \alpha))
\]
is smooth and it is a Zariski open subset of a $\mathbb P^{2a - \sum_i k_i} \times \mathbb P^{2a' - \sum_i k_i}$-bundle over $$\Phi_{\alpha}(\mathrm{Mor}(B, S, \alpha)),$$ in the Zariski topology.
\end{prop}

As a corollary, we have

\begin{coro}[{\cite[Theorem 4.2]{Tan25}}]
The scheme $\mathrm{Mor}(B, S, \alpha)$ is rational over $k$.
\end{coro}
\begin{proof}
Note that $\mathrm{Hilb}^{[k_i(\alpha)]}(B)\cong \mathbb P^{k_i(\alpha)}$, so $U_{\bold k(\alpha)}$ is rational.
Thus our assertion follows from Proposition~\ref{prop:structuretheorem}.
\end{proof}

From now on we assume that $k = \mathbb F_q$.
Since $U_{\bold k(\alpha)}$ is a Zariski open subset of the product of projective spaces and fibers are also Zariski open subsets of products of projective spaces, we can conclude that $\#\mathrm{Mor}(B, S, \alpha)(\mathbb F_q)$ is bounded by the number of $k$-points on the product of projective spaces. This leads to the following corollary:

\begin{coro}[{\cite[Theorem 8.3]{DLTT25}}]
We have
\[
\lim_{d \to \infty} \frac{\mathsf N(\mathbb P^1, S, -K_S, d)}{q^dd^5} \leq \frac{\alpha(\mathrm{Nef}_1(S))q^2}{(1-q^{-1})^7}.
\]
\end{coro}
This gives an upper bound of the correct magnitude for $\mathsf N(\mathbb P^1, S, -K_S, d)$.

\section{Homological sieve}

Let $k$ be a perfect field.
We work on the setup established in Section~\ref{sec:birationalmoduli}.
We fix a nef class $\alpha$ and we denote $\mathrm{Mor}(B, S, \alpha)$ by $M_\alpha$.
We fix a birational morphism
\[
\rho : S \to \mathbb P^1 \times \mathbb P^1,
\]
satisfying the properties in Lemma~\ref{lemm:birationalmor}.
Let $(p_i, p'_i) \in \mathbb P^1 \times \mathbb P^1$ be the image of $E_i$.
Using the above birational morphism we have the following identification:
\[
M_\alpha = \{ [s : B\to \mathbb P^1 \times \mathbb P^1] \in \mathrm{Mor}(B, \mathbb P^1 \times \mathbb P^1, (a, a'))\, | \, \mathrm{mult}_{(p_i, p_i')}(s) = k_i \}.
\]
Let $\mathbb P^1 \times \mathbb P^1 = \mathbb P(V_1) \times \mathbb P(V_2)$ and $\ell_{i, j} \subset V_j$ be the $1$-dimensional subspace corresponding to $p_i, p'_i$.
We consider the following space of sections:
\[
\widetilde{M}_\alpha =\{(s, t) \in \rH^0(B, V_1\otimes \mathcal O(a)) \oplus \rH^0(B, V_2\otimes \mathcal O(a')) \, | \, \text{ $(s, t)$ satisfies $(*)$ }\},
\]
where $(*)$ means that
\begin{itemize}
\item $s = (s_1, s_2)$ has no common root and $t$ also has no common root; and
\item the length of $(s, t)^*(\ell_{i, 1} \oplus \ell_{i, 2})$ is given by $k_i$.
\end{itemize}
Such an $(s, t)$ uniquely determines $[s : B\to \mathbb P^1 \times \mathbb P^1]  \in M_\alpha$, so we have a morphism
\[
\widetilde{M}_\alpha \to M_\alpha,
\]
which is realized as a $\mathbb G_m^2$-torsor over $M_\alpha$.

Recall that we have a smooth morphism
\[
\widetilde{M}_\alpha \to M_\alpha \to U_{\bold k}.
\]
Let $w = (w_i) \in U_{\bold k}$. Then the fiber $\widetilde{M}_{\alpha, w}$ is realized as a Zariski open subset of the vector space
\[
E_w =\rH^0(B, V_1\otimes\mathcal O(a) \oplus V_2 \otimes O(a'))_w,
\]
where $\rH^0(B, V_1\otimes\mathcal O(a) \oplus V_2 \otimes O(a'))_w$ is the space of sections $(s, t)$ such that $w_i \subset (s, t)^*(\ell_{i, 1} \oplus \ell_{i, 2})$ for any $i$.
These vector spaces $E_w$ are bundled as $E \to U_{\bold k}$ and $\widetilde{M}_\alpha$ is realized as a Zariski open subset $\widetilde{M}_\alpha \subset E$.
Then our goal is to understand the number of $k$-points on 
\[
E \setminus \widetilde{M}_\alpha,
\]
when $k = \mathbb F_q$.
A key to this is the inclusion-exclusion principle and the method of bar complexes as developed in \cite{DT24} and \cite{DLTT25}.

Following \cite[Section 5]{DLTT25}, we view the following set
\[
Q = \{ V = V_1 \oplus V_2, \ell_{i, 1} \oplus \ell_{i, 2} \, (i = 1, \cdots, 4), \ell_{i, j}, \, (i = 1, \cdots, 4, j = 1, 2), 0\},
\]
as a poscheme over $k$ (a $k$-scheme with a poset structure) where the poset structure is given by inclusion of subspaces.
Note that $Q$ admits meets, i.e., for any $p, q \in Q$, $p\wedge q = p \cap q \in Q$.
We also consider the following poscheme:
\[
\mathrm{Hilb}(B) = \{ \emptyset\} \sqcup \{B\} \sqcup \bigsqcup_{n = 1}^\infty \mathrm{Hilb}^{[n]}(B),
\]
where the poset structure is again given by inclusion of subschemes.
Then we define another poscheme
\[
\mathrm{Hilb}(B)^Q = \{ [x : Q \to \mathrm{Hilb}(B)] \in \mathrm{Mor}(Q, \mathrm{Hilb}(B)) \, |\, \text{$x$ satisfies $(*)$} \},
\]
where $(*)$ means
\[
p \leq q \in Q \implies x_p \subset x_q \text{ and } \,  x^{-1}(B) = \{V\}.
\]
Note that one can naturally embed $U_{\bold k}$ into $\mathrm{Hilb}(B)^Q$ by $w = (w_i) \mapsto x$,
where $x$ is defined by
\[
x_q =
\begin{cases}
w_i & \text{ if $q = \ell_{i, 1}\oplus \ell_{i, 2}$}\\
\emptyset & \text{otherwise}. 
\end{cases}
\]
We also define the Zariski open subscheme
\[
Q^{JB} = \{ x \in \mathrm{Hilb}(B)^Q\, |\, x_{p\wedge q} = x_p \cap x_q \text{ for any $p, q \in Q$} \} \subset \mathrm{Hilb}(B)^Q.
\]
Elements in this scheme are called saturated.

Let $\mathcal V = V_1 \otimes \mathcal O(a) \oplus V_2 \otimes \mathcal O(a')$ be a vector bundle over $B$. For any $q \in Q$, let $\mathcal K_q \subset \mathcal V$ be the corresponding subbundle over $B$.
For any $(w < x) \in (U_{\bold k} < \mathrm{Hilb}(B)^Q)$, we consider the subspace
\[
Z_{w < x} = \{ (s, t) \in \rH^0(B, \mathcal V) \, |\, \text{for any $q \in Q$}, x_q \subset (s, t)^{*}(\mathcal K_q)\} \subset E_w.
\]
These subspaces are bundled as a closed subscheme
\[
Z \subset (U_{\bold k} < \mathrm{Hilb}(B)^Q) \times_{U_{\bold k}} E.
\]
This is called the stratification of $E$. Note that we have
\[
E \setminus \widetilde{M}_\alpha = \mathrm{im} (Z \to E).
\]
When $k = \mathbb F_q$, by Grothendieck--Lefschetz trace formula, our goal is to understand the Galois modules
\[
\rH^i_{\text{\'et}, c}(\mathrm{im}(Z \to E)_{\overline{k}}, \mathbb Q_\ell).
\]
To this end, we stratify $(U_{\bold k} < \mathrm{Hilb}(B)^Q)$ into the disjoint union of locally closed subsets $\mathcal N_T$ where $T$ runs over all combinatorial types of elements of $(U_{\bold k} < \mathrm{Hilb}(B)^Q)$.
Then we truncate the poset: we choose a closed subscheme $P \subset (U_{\bold k} < \mathrm{Hilb}(B)^Q)$ which is a finite union of $\mathcal N_T$ and it is downward closed and saturation closed.

The bar complex $B(P, Z)$ is a simplicial scheme, i.e., a contravariant functor, from the category $\Delta$ of non-empty finite sets in $\mathbb N$ with non-decreasing maps to the category of separated schemes, which assigns $[n] = \{ 0, 1, \cdots, n\}$ to the scheme
\[
B(P, Z)([n]) = \{ w < x_0 \leq \cdots \leq x_n, z \in Z_{w < x_n} \, |\, (w < x_i) \in P\}.
\]
The following theorem has been proved over $\mathbb C$ by Das--Tosteson and we established this over an arbitrary perfect field:
\begin{theo}[{\cite[Theorem 5.9]{DT24} and \cite[Theorem 1.7]{DLTT25}}]
Let $I \in \mathbb N$. Suppose that
\begin{itemize}
\item $P \to U_{\bold k}$ is proper;
\item for every pair $(w < x) \in P(\overline{k})$ with $x$ saturated and every $y \in Q^{JB}(\overline{k})$ with $x \prec y$, the subspace $Z_{w < y}$ has expected dimension; and
\item $P$ contains every $\mathcal N_T$ such that the expected cohomological function $\kappa$ satisfies $\kappa(T) \leq I$.
\end{itemize}
Then for all $i > 2\dim E - I - 2$, the map
\[
\rH^i_{\textnormal{\'et}, c}(\mathrm{im}(Z \to E)_{\overline{k}}, \mathbb Q_\ell) \to \rH^i_{\textnormal{\'et}, c}(B(P, Z)_{\overline{k}}, \mathbb Q_\ell)
\]
is an isomorphism of Galois modules.
\end{theo}
The assumptions of this theorem are verified by using the description of the moduli space $M_{\alpha}$ mentioned before with
\[
I = \left\lfloor \frac{1}{8}\min\{2a + 1 - \sum_i k_i, 2a' + 1 - \sum_i k_i\} - \frac{1}{2} \right\rfloor.
\]
Let us assume that $k = \mathbb F_q$.
By the Grothendieck--Lefschetz trace formula, our goal is to understand, up to an error term, 
\[
\sum_{i \geq 4a + 4a' -2\sum_i k_i -I +2} (-1)^i\mathrm{Tr}(\mathrm{Frob}\curvearrowright \rH^i_{\text{\'et}, c}(B(P, Z)_{\overline{\mathbf k}}, \mathbb Q_\ell)).
\]
Then using a spectral sequence associated to a stratification of $B(P, Z)$ and the Grothendieck--Lefschetz trace formula for simplicial schemes, the above quantity is equal to, up to an error term,
\[
-q^{2a + 2a' + 4} \sum_{(w < x) \in (U_{\bold k} < Q^{JB})(k)} \mu_k(w, x)q^{-\gamma(x)},
\]
where $\mu_k$ is the M\"obius function for the poset $Q^{JB}(k)$ and $\gamma(x)$ is the expected codimension of the incident condition imposed by $x$. Putting altogether, the quantity $\#M_{\alpha}(k)$ is equal to
\[
q^{2a + 2a' + 4} \sum_{(w \leq x) \in (U_{\bold k} \leq Q^{JB})(k)} \mu_k(w, x)q^{-\gamma(x)},
\]
up to an error term.
\section{The virtual height zeta function}

\subsection{The virtual height zeta function}
As we explained, we have
\[
\#\widetilde{M}_\alpha(k) \sim q^{2a+2a' + 4} \sum_{(w \leq x) \in (U_{\bold k} \leq Q^{JB})(k)} \mu_k(w, x)q^{-\gamma(x)},
\]
as $a, a', k_i, I \to \infty$.
On the other hand, the expectation is
\begin{align*}
\#\widetilde{M}_\alpha(k) &\sim (q-1)^2\tau_{-K_S}(S)q^{2a+2a' - \sum_i k_i}\\
& = (1-q^{-1})^{-4}\left(\prod_{c \in |B|}(1-q^{-|c|})^6(1+ 6q^{-|c|} + q^{-2|c|}) \right)q^{2a+2a' - \sum_i k_i + 4},
\end{align*}
as $a, a', k_i, I \to \infty$.

To this end, we consider the following virtual height zeta function:
\[
\mathsf Z(\bold t) = \sum_{k_1, \cdots, k_4 = 0}^\infty q^{\sum_{i = 1}^4k_i}\left(\sum_{(w \leq x) \in (U_{\bold k} \leq Q^{JB})(k)} \mu_k(w, x)q^{-\gamma(x)}\right)t_1^{k_1}t_2^{k_2}t_3^{k_3}t_4^{k_4}.
\]
It follows from \cite[Lemma 7.7]{DLTT25} that $\mu_k$ is multiplicative. This implies that the above zeta function becomes an Euler product:
\begin{align*}
&\prod_{c \in |B|}\Big(1-6q^{-2|c|} + 8q^{-3|c|}-3q^{-4|c|}\\ &+ \sum_{i = 1}^4\sum_{d_i =1}^\infty (qt_i)^{|c|d_i}(q^{-2d_i|c|} - 2q^{-(2d_i + 1)|c|} + 2q^{-(2d_i+3)|c|} - q^{-(2d_i+4)|c|}) \Big).
\end{align*}
Using this expression, one can prove that we have
\[
\lim_{t_i \to 1} \prod_{i = 1}^4 (1-t_i)\mathsf Z(\bold t) = (1-q^{-1})^{-4}\prod_{c \in |B|}(1-q^{-|c|})^6(1 + 6q^{-|c|} + q^{-2|c|}).
\]
See \cite[Proposition 8.7]{DLTT25} for more details.
Using this and Abel summation techniques, one can prove that 
\begin{align*}
\#\widetilde{M}_\alpha(k) \sim (1-q^{-1})^{-4}\left(\prod_{c \in |B|}(1-q^{-|c|})^6(1+ 6q^{-|c|} + q^{-2|c|}) \right)q^{2a+2a' - \sum_i k_i + 4},
\end{align*}
as $a, a', k_i, I \to \infty$. See \cite[Section 8.6]{DLTT25} for more details.
Theorem~\ref{theo:main} follows from this statement.

\subsection{The error term estimates}

Finally let us mention the error term for our asymptotic formula. A key to this is the following theorem by Sawin--Shusterman:

\begin{theo}[Sawin--Shusterman, {\cite[Theorem B.1]{DLTT25}}]
Let $k$ be a separably closed field and $X$ be a projective variety defined over $k$.
Let $H$ be an ample divisor on $X$. Let $\mathrm{Mor}(\mathbb P^1, X, e)$ be the morphism scheme parametrizing rational curves of $H$-degree $e$ on $X$. Then there exists a constant $C$ such that for  $\ell$ a prime invertible in $k$, we have
\[
\sum_{i = 0}^\infty \dim \, \rH_{\textnormal{\'et}, c}^i (\mathrm{Mor}(\mathbb P^1, X, e), \mathbb Q_\ell) \leq C^{e + 1}.
\]
The constant $C$ only depends on $\dim \rH^0(X, \mathcal O(H))$ and $H^{\dim \, X}$.
\end{theo}
When $X$ is a smooth quartic del Pezzo surface, $C$ can be put as $2^{32}$. To beat this error term, we need to assume that $q^\epsilon > C$ in Theorem~\ref{theo:main}.

\section{Homological stability}

Finally we explain the applications of homological sieve method to Cohen--Jones--Segal conjecture, which was originally developed in \cite{DT24} and further extended in \cite{DLTT25} and \cite{TT25}. This was the original motivation of \cite{DT24}. Cohen--Jones--Segal conjecture, which was originally developed around 1990s in, e.g., \cite{Segal}, \cite{Guest}, and \cite{CJS00}, is the following conjecture: let $X$ be a smooth Fano variety defined over $\mathbb C$. We fix base points on $X$ and $\mathbb P^1$, and we denote the morphism scheme of pointed morphisms $s : \mathbb P^1 \to X$ of class $\alpha$ by
\[
\mathrm{Mor}_*(\mathbb P^1, X, \alpha).
\]
Similarly we also denote the space of pointed continuous maps of class $\alpha$ in the Euclidean topology by
\[
\mathrm{Top}_*(\mathbb P^1(\mathbb C), X(\mathbb C), \alpha),
\]
which is equipped with compact topology. This is generally an infinite dimensional space like a Banach space. We have the inclusion
\[
\mathrm{Mor}_*(\mathbb P^1, X, \alpha)(\mathbb C) \hookrightarrow \mathrm{Top}_*(\mathbb P^1(\mathbb C), X(\mathbb C), \alpha),
\]
and this realizes $\mathrm{Mor}_*(\mathbb P^1, X, \alpha)(\mathbb C)$ as a subspace of $\mathrm{Top}_*(\mathbb P^1(\mathbb C), X(\mathbb C), \alpha)$. Cohen--Jones--Segal conjecture predicts that the homotopy type of $\mathrm{Mor}_*(\mathbb P^1, X, \alpha)(\mathbb C)$ approximates the homotopy type of $\mathrm{Top}_*(\mathbb P^1(\mathbb C), X(\mathbb C), \alpha)$ as the degree of $\alpha$ tends to $\infty$. A precise statement is given in terms of stabilization of homology and homotopy groups. Here is our result:
\begin{theo}[{\cite[Theorem 1.6]{DLTT25}}]
Let $S$ be a smooth del Pezzo surface of degree $4$ over $\mathbb C$. There exist a dense open subset cone $U \subset \Nef_1(S)$, a non-negative, rational homogeneous, continuous piecewise linear function $\ell : \Nef_1(S) \to \mathbb R$ which is positive on $U$, and a constant $c \in \mathbb R_{\geq 0}$ such that for any $\alpha \in U$, the homomorphisms
\[
\rH_{i}^{\mathrm{sing}}(\mathrm{Mor}_*(\mathbb P^1, X, \alpha)(\mathbb C), \mathbb Z) \to \rH_{i}^{\mathrm{sing}}(\mathrm{Top}_*(\mathbb P^1(\mathbb C), X(\mathbb C), \alpha), \mathbb Z),
\]
induced by the inclusion, are isomorphisms assuming $i \leq \ell(\alpha)-c$.
\end{theo}

Such a statement was originally obtained for degree $5$ del Pezzo surfaces using the homological sieve in \cite{DT24}. We extended this result to quartic del Pezzo surfaces.

A key to the proof of this theorem is that, as we explained, we realized $\mathrm{Mor}_*(\mathbb P^1, X, \alpha)$ as a Zariski open subset of a bundle of the products of projective spaces. To this end, we consider the space of pointed continuous positive maps
\[
\mathrm{Top}^+_*(\mathbb P^1(\mathbb C), X(\mathbb C), \alpha),
\]
which parametrizes continuous maps which have proper positive intersections with exceptional divisors, and \cite{DT24} proves that the inclusion
\[
\mathrm{Top}^+_*(\mathbb P^1(\mathbb C), X(\mathbb C), \alpha) \hookrightarrow \mathrm{Top}_*(\mathbb P^1(\mathbb C), X(\mathbb C), \alpha)
\]
is a homotopy equivalence. Then \cite{DT24} further develops a semi-topological model $\mathcal M^+_\alpha \subset \mathrm{Top}^+_*(\mathbb P^1(\mathbb C), X(\mathbb C), \alpha)$ which is again a homotopy equivalence. Then we realize $\mathcal M^+_\alpha$ as an open subset of a bundle of the products of projective spaces of Banach spaces. Then the complements of these two open subsets in respective bundles are stratified by the same poscheme $$(U_{\bold k(\alpha)} < \mathrm{Hilb}(\mathbb P^1)^Q).$$
This induces two spectral sequences encoding those inclusion-exclusion information which are compatible. In this way one can prove that homology groups of two spaces are isomorphic to each other. See \cite[Section 9]{DLTT25} for more details.

\bibliographystyle{amsalpha}
\bibliography{RIMS}

\end{document}